\documentclass[11pt,oneside]{article}
\usepackage{amssymb}
\usepackage{amsmath}
\usepackage{amsthm}

\usepackage{wasysym}
\usepackage{hyperref}
\usepackage{color}
\usepackage{fmtcount}
\usepackage{tikz}
\usetikzlibrary{scopes,shapes}
\usetikzlibrary{decorations.pathreplacing}
\include{longdiv}
\usepackage{multirow}
\usepackage[top=1in, bottom=1in, left=1.5in, right=1in]{geometry}
\usepackage{titlesec}
\usepackage{tocloft}
\usepackage{setspace}
\usepackage{tocloft}
\usepackage{chngcntr}

\hypersetup{colorlinks=true, linkcolor=red}

\theoremstyle{plain}
\newtheorem{theorem}{Theorem}

\newtheorem{lemma}[theorem]{Lemma}

\newtheorem{proposition}[theorem]{Proposition}

\newtheorem{corollary}[theorem]{Corollary}

\newtheorem{question}{Question}
\newtheorem{problem}[question]{Problem}

\usepackage{hyperref}
\usepackage{cleveref}
\usepackage{thmtools}
\usepackage{thm-restate}
\usepackage{setspace}
\usepackage{verbatim}
\usepackage{mathrsfs}
\usetikzlibrary{arrows, automata, positioning}
\usepackage{adjustbox,tikz,calc,graphics,standalone}
\usepackage[makeroom]{cancel}

\usepackage{enumerate}

\usetikzlibrary{shapes.misc,calc,intersections,patterns,decorations.pathreplacing}
\usetikzlibrary{arrows,shapes,positioning}
\usetikzlibrary{decorations.markings}
\usetikzlibrary{graphs,graphs.standard,chains,fit}

\newcommand{\comp}[1]{\overline{#1}}

\newcommand{\ldim}[1]{\operatorname{ldim}\left(#1\right)}

\iftrue
\newcommand{\fdim}[1]{\operatorname{dim^\star}\left(#1\right)}
\newcommand{\fldim}[1]{\operatorname{ldim^\star}\left(#1\right)}
\newcommand{\ftdim}[2]{\operatorname{dim^\star_{#1}}\left(#2\right)}
\newcommand{\fltdim}[2]{\operatorname{ldim^\star_{#1}}\left(#2\right)}
\newcommand{\ftwodim}[1]{\operatorname{dim^\star_2}\left(#1\right)}
\newcommand{\fltwodim}[1]{\operatorname{ldim^\star_2}\left(#1\right)}
\fi

\newcommand{\pdim}[1]{\operatorname{dim}\left(#1\right)}
\newcommand{\twodim}[1]{\operatorname{dim}_2\left(#1\right)}
\newcommand{\ltwodim}[1]{\operatorname{ldim}_2\left(#1\right)}
\newcommand{\ltdim}[2]{\operatorname{ldim}_{#1}\left(#2\right)}
\newcommand{\tdim}[2]{\operatorname{dim}_{#1}\left(#2\right)}
\newcommand{\lbc}[1]{\operatorname{lbc}\left(#1\right)}
\newcommand{\ldc}[1]{\operatorname{ldc}\left(#1\right)}

\newcommand{\cube}[2]{\mathcal{Q}^{#1}_{#2}}

\DeclareMathOperator{\N}{\mathbb{N}}
\DeclareMathOperator{\R}{\mathbb{R}}

\DeclareMathOperator{\Expect}{\mathbb{E}}
\DeclareMathOperator{\dom}{\operatorname{dom}}

\begin{document}

\def\ver{1}
\def\comp{1}
\def\compstrict{2}
\def\tikh{3}

\title{Local $t$-dimension}
\author{David Lewis \\{\small Department of Mathematical Sciences} \\ {\small University of Memphis} \\ {\small\tt davidcharleslewis@outlook.com}}

\maketitle

\begin{abstract}
In this note, we introduce a new poset parameter called local $t$-dimension. We also discuss the fractional variants of this and other dimension-like parameters.
\end{abstract}

\section{Introduction}

In this note, we introduce a new poset parameter called the \emph{local $t$-dimension}, where $t$ is an integer greater than or equal to $2$. This new parameter combines the notions of $t$-dimension, first defined by Nov\'{a}k~\cite{novak}, with local dimension, first defined by Ueckerdt~\cite{ueckerdt}, which are both variants of the notion of order dimension introduced by Dushnik and Miller~\cite{dm}.

In this note, all posets are assumed to be nonempty by definition. All logarithms are base $2$ unless otherwise specified. For each integer $n$, boldface $\mathbf{n}$ denotes an $n$-element chain and $A_n$ denotes an $n$-element antichain.

The $n$-dimensional Boolean lattice, defined as the set of all subsets of $[n]$ ordered by inclusion, is denoted $\cube{n}{}$. The suborder of $\cube{n}{}$ induced by the $\ell^\textrm{th}$ and $k^\textrm{th}$ layers (i.e., $\big([n]^{(\ell)}\cup[n]^{(k)},\subseteq\big)$) is called $\cube{n}{\ell,k}$.

We will make use of some nonstandard definitions of dimension-theoretic concepts. However, all our definitions are easily seen to be equivalent to the standard ones.

 Given a poset $P$, a \emph{local realiser} of $P$ is a set $\mathcal{L}$ of monotone partial functions from $P$ to a chain $C$ such that, for every $x$ and $y$ in $P$ with $x \not\geq y$, there is a partial function $f\in\mathcal{L}$ such that $x,y\in\dom(f)$ and $f(x) < f(y)$. A local realiser of $P$ is called a \emph{realiser} of $P$ if all of its elements are total functions, i.e., functions whose domains are all of $P$. Given an integer $t \geq 2$, a local realiser $\mathcal{L}$ of $P$ is called a \emph{local $t$--realiser} of $P$ if the codomain of every partial function in $\mathcal{L}$ is the $t$-element chain $\mathbf{t}$. A $t$--realiser of $P$ is a local $t$--realiser that is also a realiser.


 The \emph{dimension} of a poset $P$, denoted $\pdim{P}$. is defined as the minimum cardinality of a realiser of $P$. For any integer $t\geq 2$, the \emph{$t$--dimension} of $P$, denoted $\tdim{t}{P}$, is the minimum cardinality of a $t$--realiser of $P$. The $t$--dimension of $P$ is monotone decreasing in $t$, and, for finite posets $P$, the minimum value of $\tdim{t}{P}$ over all values of $t$ is equal to $\pdim{P}$. The most interesting case of $t$--dimension is $2$--dimension, as $\twodim{P}$ is equal to the smallest $d$ such that $P$ embeds into $\cube{d}{}$ as a suborder.

 Given a local realiser $\mathcal{L}$ of $P$ and a point $x\in P$, the multiplicity of $x$ in $\mathcal{L}$, denoted $\mu_\mathcal{L}(x)$, is defined as the number of partial functions $f\in\mathcal{L}$ such that $x\in\dom(f)$. The local dimension of $P$, denoted $\ldim{P}$, is the minimum over all local realisers $\mathcal{L}$ of $P$ of $\max\left\{\mu_\mathcal{L}(x) : x\in P\right\}$. For any integer $t\geq 2$, the local $t$--dimension of $P$, denoted $\ltdim{t}{P}$, is defined as the minimum over all local $t$--realisers $\mathcal{R}$ of $P$ of $\max\left\{\mu_\mathcal{R}(x) : x\in P\right\}$. Similar to the case with $t$--dimension, local $t$--dimension is monotone decreasing in $t$, and the minimum value of $\ltdim{t}{P}$ over all values of $t$ is (when $P$ is finite) equal to $\ldim{P}$. As with $t$--dimension, we will usually consider the case where $t=2$.

 The following inequalities follow immediately from the definitions, and hold for all posets $P$ and all choices of $t$.
\begin{equation}\label{ineq:1}
\ldim{P} \leq \ltdim{t}{P} \leq \tdim{t}{P},
\end{equation}
\begin{equation}\label{ineq:2}
\ldim{P} \leq \pdim{P} \leq \tdim{t}{P}.
\end{equation}
As we will see, dimension and local $t$--dimension are incomparable, and there exist posets of bounded dimension and arbitrarily large local $t$--dimension. However, local $t$--dimension is bounded below by a logarithmic function of dimension:
\begin{align}\label{ineq:logbound}
\ltdim{t}{P} \geq \log_t \left(2\pdim{P}-1\right).
\end{align}

 A poset parameter $f$ is called \emph{monotone} if, for every poset $Q$ and every suborder $P$ of $Q$, we have $f(P) \leq f(Q)$. It is called \emph{subadditive} if, for all posets $P$ and $Q$, we have $f(P\times Q) \leq f(P) + f(Q)$. Dimension, local dimension, and $t$--dimension are all monotone and subadditive; see \cite{dm}, \cite{localdim}, and \cite{trotter}, respectively, for the proofs. We will now show that this is true for local $t$--dimension as well.

 To prove monotonicity, let $\mathcal{R}$ be a local $t$--realiser of $Q$ and $P$ a suborder of $Q$. Since $\mathcal{R}\mid_P = \{f\mid_P:f\in\mathcal{R}\}$ is a local $t$--realiser of $P$ whose maximum multiplicity is at most that of $\mathcal{R}$, $\ltdim{t}{P}\leq \ltdim{t}{Q}$.

 For subadditivity, let $P$ and $Q$ be posets and let $\mathcal{R}$ and $\mathcal{S}$ be local $t$--realisers of $P$ and $Q$ respectively. Let $\mathcal{T} = \left\{f\circ\pi_P : f\in\mathcal{R}\right\} \cup \left\{f\circ\pi_Q : f\in\mathcal{S}\right\}$, where $\pi_P$ and $\pi_Q$ are the projection maps from $P\times Q$ onto $P$ and $Q$ respectively. Then $\mathcal{T}$ is a local $t$--realiser of $P\times Q$ and, for every $(x,y)$, $\mu_\mathcal{T}(x,y) = \mu_\mathcal{R}(x) + \mu_\mathcal{S}(y)$.

\section{Bounds on local $t$--dimension}
 Recall that $\tdim{t}{P}$ is equal to the smallest cardinal $d$ such that $P$ embeds into $\mathbf{t}^d$ as a suborder. We therefore have the trivial bound $\tdim{t}{P} \geq \log_t |P|$ due to the pigeonhole principle. Our first theorem shows that the same bound holds for local $t$--dimension.
\begin{theorem}\label{thm:lowerbound}
For every poset $P$ with cardinality $n$, $\ltdim{t}{P} \geq \log_t n$.
\end{theorem}
\begin{proof}
Let $P$ be a poset of cardinality $n$ and let $\mathcal{R}$ be a local $t$--realiser of $P$. For each pair of distinct elements $x,y\in P$, either $x\not\geq y$ or $y\not\geq x$. Either way, there is a partial function $f\in\mathcal{R}$ such that $f(x)\neq f(y)$. For each $f\in\mathcal{R}$, let $G_f$ be a graph with vertex set $\dom(f)$ and edge set $\{xy : f(x)\neq f(y)\}$. Clearly each $G_f$ is a complete $t$-partite graph, and the set $\left\{G_f : f\in\mathcal{R}\right\}$ is an edge cover of $K_n$.

 Now, for each $f\in\mathcal{R}$, let $U_f$ be one of the $t$ classes of $G_f$, chosen independently and uniformly at random, and let $U$ be the intersection of all the $U_f$'s. For each edge $xy$ of $K_n$, there is an $f\in \mathcal{R}$ such that $x$ and $y$ are in different classes of $G_f$, so $x$ and $y$ cannot both be in $U_f$. Therefore $U$ is an independent set and so has at most one element. Now, for each $x\in P$, the probability that $x\in U$ is $t^{-\mu_{\mathcal{R}}(x)}$, so $\Expect[|U|] = \sum_{x\in P}t^{-\mu_{\mathcal{R}}(x)} \leq 1$. Now let $\mu = \frac{1}{n}\sum_{x\in P}\mu_\mathcal{R}(x)$ (i.e., the average multiplicity of $\mathcal{R}$). By convexity, $nt^{-\mu} \leq \sum_{x\in P}t^{-\mu_{\mathcal{R}}(x)} \leq 1$, and hence $\mu\geq \log_t n$.

 Note that the case $t=2$ was proved by Hansel~\cite{hansel}; see also Bollob\'{a}s and Scott~\cite{bs}.
\end{proof}

 This bound is clearly sharp, as $\mathbf{t}^n$ has cardinality $t^n$ and local $t$--dimension at most (and hence exactly) $n$. Hiraguchi~\cite{hiraguchi} proved that a poset of dimension $n\geq 3$ has cardinality at least $2n-1$, which implies inequality~\ref{ineq:logbound}.

 The next proposition shows that chains also have the smallest local $t$--dimension possible given their cardinality. This contrasts with $t$--dimension; it's a simple exercise to prove that $\tdim{t}{\mathbf{n}} = \left\lceil\frac{n-1}{t-1}\right\rceil$.

\begin{proposition}\label{prop:ltdimchain}
For all $n\in\N$, $\ltdim{t}{\mathbf{n}} = \lceil \log_t n \rceil$.
\end{proposition}
\begin{proof}
Obviously, $\ltdim{t}{\mathbf{1}} = 0$. Now let $\mathcal{R}$ be a local $t$--realiser of $\mathbf{n}$. We construct a local $t$--realiser of $\mathbf{tn}$ splitting $\mathbf{tn}$ into $t$ equal segments and taking a copy of $\mathcal{R}$ covering each segment, as well as a total function from $\mathbf{tn}$ to $\mathbf{t}$ that sends the $i^\textrm{th}$ segment to $i$, for each $i\in \mathbf{t}$. This shows that $\ltdim{t}{\mathbf{tn}} \leq \ltdim{t}{\mathbf{n}}+1$, and hence by induction $\ltdim{t}{\mathbf{n}} \leq \lceil\log_t n\rceil$ for all $n$. The matching lower bound follows from Theorem~\ref{thm:lowerbound}.
\end{proof}

\begin{corollary}\label{cor:ldimbound}
For every poset $P$ with cardinality $n$ and every integer $t\geq 2$,
\[\ltdim{t}{P} \leq \lceil\log_t n\rceil\ldim{P}.\]
For every poset $P$ and every pair of integers $t\geq s \geq 2$,
\[\ltdim{t}{P} \leq \lceil\log_t s\rceil\ltdim{s}{P}.\]
\hfill\qedsymbol
\end{corollary}

 For antichains, a similar argument shows that $\ltdim{t}{A_n} \leq 2\lceil\log_t n\rceil$. In the case $t=2$, we can do better. It follows from Sperner's theorem that
\begin{equation}
\ltwodim{A_n} \leq \twodim{A_n} = \min\left\{m : \binom{m}{\lfloor m/2\rfloor} \geq n\right\} .
\end{equation}
The corresponding upper bound follows from a theorem of Bollob\'{a}s and Scott~\cite{bs}.

\begin{proposition}
For all $n\in\N$,
\[\ltwodim{A_n} \geq \min\left\{m : \binom{m+1}{\lfloor (m+1)/2\rfloor} \geq n+1\right\}.\]
\end{proposition}
\begin{proof}
A local $2$--realiser of $A_n$ is a set $\mathcal{R}$ of partial functions from $[n]$ to $\{0,1\}$ such that, for every ordered pair $(x,y)\in [n]^2$ with $x\neq y$, there is an $f\in\mathcal{R}$ such that $f(x) = 0$ and $f(y) = 1$. Such a set is also known as a \emph{strongly separating system} on $[n]$. Bollob\'{a}s and Scott~\cite{bs} proved that, for every strongly separating system $\mathcal{R}$ on $[n]$, the sum of the cardinalities of the domains of the functions in $\mathcal{R}$ is at least $kn$, where $k$ is the smallest integer such that
\begin{equation}
\binom{k+1}{\lfloor (k+1)/2\rfloor} \geq n+1.
\end{equation}
It follows that there exists an element of $A_n$ whose multiplicity in $\mathcal{R}$ is at least $k$.
\end{proof}

 Using Stirling's inequality to estimate the upper and lower bounds, it follows that $\ltwodim{A_n} = \log n + \frac{1}{2}\log\log n + O(1)$. Using more precise estimates, one can show that the $O(1)$ term is at most $2$.

\section{Local $2$--dimension and complete bipartite edge-coverings of graphs}
 Let $P$ be a two-level poset with minimal elements $A$ and maximal elements $B$, with $A$ and $B$ disjoint. The bipartite imcomparability graph of $P$ is the graph with vertex set $A\cup B$ and edge set $\{ab : a\in A, b\in B, a \not< b\}$.

 In \cite{localdim}, Kim, Martin, Masa{\v{r}}{\'\i}k, Shull, Smith, Uzzell, and Wang showed that the local dimension of a two-level poset $P$ is essentially the same (up to an additive constant) as a certain graph parameter, called the \emph{local difference graph covering number}, of the bipartite incomparability graph of $P$. In this section, we will show that local $2$-dimension has a similar connection with another graph parameter, namely the local complete bipartite covering number.

 Let $G$ be a graph. A complete bipartite edge-covering of $G$ is a set of complete bipartite subgraphs of $G$, the union of whose edge sets is the edge set of $G$. Given an edge cover $\mathcal{C}$ of $G$ and a vertex $v\in G$, the multiplicity of $v$ in $\mathcal{C}$, denoted $\mu_\mathcal{C}(v)$, is the number of subgraphs in $\mathcal{C}$ whose vertex sets contain $v$. The local complete bipartite covering number of $G$, denoted $\lbc{G}$, is defined as the minimum of $\max\left\{\mu_\mathcal{C}(v) : v\in V(G)\right\}$ over all complete bipartite edge-coverings $\mathcal{C}$ of $G$.

 Note that a random bipartite graph with classes of cardinality $n$ has local complete bipartite covering number $\Omega(n/\log n)$ with high probability, so in the following theorem, $\log|A|$ is typically much smaller than $\lbc{G}$.

\begin{theorem}\label{thm:bipinc}
Let $P$ be a two-level poset $P$ with minimal elements $A$ and maximal elements $B$, and let $G$ be the bipartite incomparability graph of $P$. Assume without loss of generality that $|A| \geq |B|$. Then
\[\lbc{G} \leq \ltwodim{P} \leq \lbc{G} + \log |A| + \tfrac{1}{2}\log\log |A| + 3.\]
\end{theorem}
\begin{proof}
First we show that $\lbc{G} \leq \ltwodim{P}$. To this end, let $\mathcal{R}$ be a local $2$--realiser of $P$. For each partial function $f\in\mathcal{R}$, let $B_f$ be the complete bipartite graph with classes $f^{-1}(1)\cap A$ and $f^{-1}(0)\cap B$. Let $\mathcal{C} = \left\{B_f : f\in\mathcal{R}\right\}$. Now $\mathcal{C}$ is a complete bipartite edge-covering of $G$, and, for each $v\in P$, $\mu_\mathcal{C}(v) = \mu_\mathcal{R}(v)$.

 Now we show that $\ltwodim{P} \leq \lbc{G} + \log |A| + \tfrac{1}{2}\log\log |A| + 3$. Let $\mathcal{C}$ be a complete bipartite edge-covering of $G$. For each $B\in\mathcal{C}$, define a partial function $f_B$ with domain $V(B)$ by $f_B(a) = 1$ if $a\in A$ and $f(b) = 0$ if $b\in B$. Each such partial function is monotone and, for each $a\in A$, $b\in B$ with $a$ and $b$ incomparable, there is a $B\in\mathcal{C}$ such that $f_B(b) < f_B(a)$. Now define a function $f$ with domain $P$ by $f(a) = 0$ if $a\in A$, $f(b) = 1$ if $b\in B$. Finally, let $\mathcal{R}$ and $\mathcal{S}$ be local $2$--realisers of the antichains $A$ and $B$ respectively. The set $\mathcal{T} = \left\{B_f : B\in\mathcal{C}\right\}\cup\mathcal{R}\cup\mathcal{S}\cup\{f\}$ is a local $2$--realiser of $P$. For each $a\in A$, $\mu_\mathcal{T}(a) = \mu_\mathcal{C}(a) + \mu_\mathcal{R}(a) + 1$ and, for each $b\in B$, $\mu_\mathcal{T}(b) = \mu_\mathcal{C}(b) + \mu_\mathcal{S}(b) + 1$. As we saw in the previous section, $\mathcal{R}$ and $\mathcal{S}$ can be chosen so that each element has multiplicity at most
\begin{equation}
\min\left\{m : \tbinom{m}{\lfloor m/2\rfloor} \geq |A|\right\} \leq \log |A| + \tfrac{1}{2}\log\log |A| + 2.
\end{equation}
\end{proof}

\begin{corollary}
Let $S_n$ be the standard example of a poset of dimension $n$, namely the suborder of the $n$--dimensional Boolean lattice consisting of all subsets of $[n]$ of cardinality $1$ and all subsets of cardinality $n-1$. For all $n\geq 2$,
\[\ltwodim{S_n} \leq \log n + \tfrac{1}{2}\log\log n + 4.\]
\end{corollary}
\begin{proof}
This follows from Theorem~\ref{thm:bipinc} and the fact that the bipartite incomparability graph of $S_n$ is a matching.
\end{proof}

 The split of a poset $P$, first defined by Kimble (see \cite{trotter}), is defined as the two-level poset $Q$ with minimal elements $P' = \{x' : x\in P\}$ and maximal elements $P'' = \{x'' : x\in P\}$, where $x' \leq y'$ if and only if $x\leq y$ in $P$.

 The following lemma is analogous to a lemma proved for local dimension by Barrera-Cruz, Prag, Smith, Taylor, and Trotter in \cite{bpstt}.

\begin{lemma}\label{lem:split}
Let $P$ be a poset with $n$ elements and let $Q$ be the split of $P$. Then
\[\ltwodim{Q} - \log n - \tfrac{1}{2}\log\log n - 3 \leq \ltwodim{P} \leq 2\ltwodim{Q} - 2.\]
\end{lemma}
\begin{proof}
Let $\mathcal{R}$ be a local $2$--realiser of $P$. For each partial function $f\in\mathcal{R}$, define a partial function $f'$ with domain $\{x' : f(x) = 1\} \cup \{x'' : f(x) = 0\}$, sending each $x'$ and each $x''$ to $f(x)$. Let $\mathcal{S}$ and $\mathcal{T}$ be local $2$--realisers of the antichains $P'$ and $P''$, and let $g$ be the total function that maps $P'$ to $0$ and $P''$ to $1$. The union $\left\{f' : f\in\mathcal{R}\right\} \cup \mathcal{S} \cup \mathcal{T} \cup \{g\}$ is a local $2$--realiser of $Q$, and $\mathcal{R}$, $\mathcal{S}$, and $\mathcal{T}$ can be chosen so that each element of $Q$ has multiplicity at most $\ltwodim{P} + \ltwodim{A_n} + 1$. Therefore, $\ltwodim{Q} \leq \ltwodim{P} + \ltwodim{A_n} + 1$.

 Now let $\mathcal{R}$ be a local $2$--realiser of $Q$.
For each $f\in\mathcal{R}$, we define a partial function $f'$ with domain $\left\{x \in P : f(x') = 1 \textrm{ or } f(x'') = 0\right\}$, mapping $x$ to $1$ if $f(x') = 1$ and $0$ if $f(x'') = 0$. Because each $f$ is monotone and $x' < x''$ for all $x\in P$, only one of these cases can be true for each $x\in\dom(f')$.
It is easy to check that $f'$ is monotone (there are four cases to consider). For each $x$ and $y$ in $P$ with $x \not\geq y$, $x'' \not\geq y'$, so there is an $f\in\mathcal{R}$ such that $f(x'') = 0$ and $f(y') = 1$, and hence $f'(x) < f'(y)$. Therefore $\mathcal{S} = \left\{f' : f\in\mathcal{R}\right\}$ is a local $2$--realiser of $P$. For each $x\in P$, there is a $g\in\mathcal{R}$ such that $g(x') = 0$ and $f(x'') = 1$, so $x\not\in\dom(g')$. It follows that $\mu_\mathcal{S}(x) \leq \left(\mu_\mathcal{R}(x') - 1\right) + \left(\mu_\mathcal{R}(x'') - 1\right)$, and hence $\ltwodim{P} \leq 2\ltwodim{Q} - 2$.
\end{proof}

For any $n\in\N$, let $H_n$ be the bipartite incomparability graph of the split of $\mathbf{n}$.
A \emph{difference graph with $n$ steps} is a graph that can be obtained from $H_n$ by a sequence of vertex duplications.
For other definitions and characterisations of difference graphs, see \cite{hammerpeledsun}.
The \emph{local difference graph covering number} of a graph $G$, denoted $\ldc{G}$, is defined in the same way as $\lbc{G}$, substituting "complete bipartite graph" with "difference graph." Since a complete bipartite graph is just a difference graph with one step, $\ldc{G} \leq \lbc{G}$ for every graph $G$. Let $P$ be a two-layer poset and $G$ its bipartite incomparability graph. Kim et al. proved the analogue of Theorem~\ref{thm:bipinc} for local dimension, which states that $\ldc{G} \leq \ldim{P} \leq \ldc{G} + 2$.

Dam\'{a}sdi, Felsner,  Gir\~{a}o, Keszegh, Lewis, Nagy, and Ueckerdt~\cite{diffgraph} proved that a difference graph with $n$ steps has local complete bipartite covering number equal to
\begin{equation}
\min\left\{k : \tbinom{2k}{k} \geq n+1\right\} = \tfrac{1}{2}\log n + \tfrac{1}{4}\log\log n + O(1).
\end{equation}
This implies a version of Corollary~\ref{cor:ldimbound} for graphs, namely, for every graph $G$ on $n$ vertices,
\begin{equation}
\lbc{G} \leq \lbc{H_{\lceil n/2\rceil}}\ldc{G} \leq \left(\tfrac{1}{2}\log n + \tfrac{1}{4}\log\log n + \tfrac{3}{2}\right)\ldc{G}.
\end{equation}

 The Erd\H{o}s-Pyber theorem~\cite{ep} states that, for every graph $G$ with $n$ vertices, $\lbc{G} = O\left(\frac{n}{\log n}\right)$. Csirmaz, Ligeti, and Tardos~\cite{clt} showed that $\lbc{G} \leq (1+o(1))\frac{n}{\log n}$. We can use this to bound the local $2$--dimension of any poset from above.

\begin{theorem}
For every poset $P$ with cardinality $n$,
\[\ltwodim{P} \leq (4+o(1))\frac{n}{\log n}.\]
\end{theorem}
\begin{proof}
Let $Q$ be the split of $P$ and let $G$ be the bipartite incomparability graph of $Q$. By Lemma~\ref{lem:split} and Theorem~\ref{thm:bipinc}, $\ltwodim{P} \leq 2\ltwodim{Q} \leq 2\lbc{G} + O(\log n)$. Since $|G| = 2n$, $\lbc{G} \leq (2+o(1))\frac{n}{\log n}$.
\end{proof}

 Kim et al.~\cite{localdim} proved that, as $n\to\infty$, there exist $n$-element posets with local dimension $\Omega\left(\frac{n}{\log n}\right)$. Of course the same is true for local $t$--dimension for every $t$. The author~\cite{lewis} improved this lower bound by a constant factor, showing that there exists an $n$-element poset with local dimension (and hence local $t$--dimension for every $t$) at least $\frac{n}{4\log 3n}$ for all $n\geq 2$.

 By a theorem of Kierstead~\cite{kierstead}, for all integers $\ell$, $k$, and $n$ with $1\leq \ell \leq k \leq n$, $\ltwodim{\cube{n}{\ell,k}} \leq \twodim{\cube{n}{\ell,k}} \leq \twodim{\cube{n}{1,k}} \leq \lceil e(k+1)^2\ln n\rceil$. By Theorem~\ref{thm:lowerbound}, this bound is the best possible up to a constant factor when $k$ is constant. Kierstead's argument can also be used to show that $\ltwodim{\cube{n}{1,2}} \leq \twodim{\cube{n}{1,2}} \leq \lceil 3\log_{27/23}n\rceil$.

\section{Fractional $t$--dimension and local $t$--dimension}
 Each of the poset parameters we have discussed can be described as the optimal solution to a certain integer program. In this section, we consider the linear programming relaxations of these programs, whose solutions are called the fractional variants of the original parameters.

 A \emph{fractional local realiser} of a poset $P$ is a function $w$ that assigns a nonnegative weight to each monotone partial function from $P$ to a chain $C$ in such a way that, for every pair $x\not\geq y$, $\sum \left\{w(f) : f(x) < f(y)\right\} \geq 1$. A fractional realiser is a fractional local realiser that assigns positive weight only to total functions, and a fractional local $t$--realiser is a fractional local realiser where the chain $C$ has $t$ elements. The fractional (local) ($t$)-dimension of a poset $P$ is the minimum over all fractional (local) ($t$)-realisers $w$ of $\max \left\{\sum_{x\in\dom(f)} w(f) : x\in P\right\}$.
Following Bir\'{o}, Hamburger, and P\'{o}r~\cite{bhp}, we denote the fractional variant of a parameter by adding a superscript $\star$ to the corresponding integer parameter.
Fractional dimension was introduced and studied by Brightwell and Scheinerman~\cite{brightwellscheinerman} and fractional local dimension by Smith and Trotter~\cite{smithtrotter}, but, as far as we know, fractional $t$--dimension and fractional local $t$--dimension have never been studied.

 Like the corresponding integer parameters, these fractional parameters are easily shown to be subadditive and monotonic. Also, Inequalities~\ref{ineq:1}~and~\ref{ineq:2} hold for the fractional variants as well.

 It is trivial to show that $\fltdim{t}{A_n} \leq \ftdim{t}{A_n} \leq \frac{2t}{t-1}$ for all $n$ and all $t$ -- just take $w$ to be the constant function $\frac{2t}{t-1}\cdot t^{-n}$ -- so fractional (local) $t$--dimension cannot be bounded below by a function of cardinality.


 We can determine the fractional $t$--dimension of a chain exactly.

\begin{theorem}\label{thm:ftdimchain}
For all integers $n \geq t \geq 2$, $\ftdim{t}{\mathbf{n}} = \frac{n-1}{t-1}$.
\end{theorem}
\begin{proof}
Let $w$ be a fractional $t$--realiser of $\mathbf{n}$. For each $x$ and $y$ in $\mathbf{n}$ such that $y$ covers $x$, $w$ must assign total weight at least $1$ to the set of monotone functions $f$ such that $f(x) < f(y)$. Conversely, each such $f$ separates at most $t-1$ covering relations. Since $\mathbf{n}$ has $n-1$ covering relations, we have
\begin{equation}
\begin{split}
\sum\limits_{\substack{f: \mathbf{n} \to \mathbf{t} \\ f\text{ monotone}}} (t-1)w(f) \ \ \geq \sum\limits_{\substack{f: \mathbf{n} \to \mathbf{t} \\ f\text{ monotone}}} \sum\limits_{\substack{x, y\in \mathbf{n} \\ y\text{ covers }x \\ f(x) < f(y)}} w(f) \ \ =
\\
\sum\limits_{\substack{x, y\in \mathbf{n} \\ y\text{ covers }}} \sum\limits_{\substack{f: \mathbf{n} \to \mathbf{t} \\ f\text{ monotone} x \\ f(x) < f(y)}} w(f) \geq n-1.
\end{split}
\end{equation}
It follows that $\ftdim{t}{\mathbf{n}} \geq \frac{n-1}{t-1}$.

 To show that $\ftdim{t}{\mathbf{n}} \leq \frac{n-1}{t-1}$ when $n > t$ (the case $n=t$ is trivial), we will define a set $F$ of montone functions from $\mathbf{n}$ to $\mathbf{t}$ such that $|F| = n-1$ and, for every pair $x,y\in \mathbf{n}$ such that $y$ covers $x$, there are exactly $t-1$ functions $f\in F$ such that $f(x) < f(y)$. Then the function $w$ that assigns weight $t-1$ to each element of $F$ and weight $0$ to each monotone function not in $F$ is a fractional $t$--realiser of $\mathbf{n}$ with total weight $\frac{n-1}{t-1}$.

 Let $F$ be the set of all monotone functions $f$ from $\mathbf{n}$ to $\mathbf{t}$ with the following properties:
\begin{enumerate}
\item $f$ is surjective;
\item for each $x\in \mathbf{t}$ that is not the top or bottom element, $|f^{-1}\{x\}| \leq 2$;
\item for all $x < y < z\in \mathbf{t}$, if $|f^{-1}(x)| \geq 2$ and $|f^{-1}(z)| \geq 2$, then $|f^{-1}(y)| \geq 2$.
\end{enumerate}
\noindent
For example, when $n = 9$ and $t = 4$, $F$ consists of the following functions:
\begin{align*}
&123456 \ \ 7 \ \ 8 \ \ 9 \\
&12345 \ \ 67 \ \ 8 \ \ 9 \\
&1234 \ \ 56 \ \ 78 \ \ 9 \\
&123 \ \ 45 \ \ 67 \ \ 89 \\
&12 \ \ 34 \ \ 56 \ \ 789 \\
&1 \ \ 23 \ \ 45 \ \ 6789 \\
&1 \ \ 2 \ \ 34 \ \ 56789 \\
&1 \ \ 2 \ \ 3 \ \ 456789.
\end{align*}

Now we will show that $F$ has the desired properties. First, denote the bottom and top elements of $\mathbf{t}$ by $\alpha$ and $\omega$ respectively, and define $A_f = |f^{-1}\{\alpha\}|$ and $\Omega_f = |f^{-1}\{\omega\}|$. If $n \geq 2t-1$, then there are $n-2t+1$ different functions $f\in F$ such that $A_f \geq 2$ and $\Omega_f \geq 2$, $t-1$ functions $f\in F$ such that $\Omega_f = 1$, and $t-1$ functions $f\in F$ such that $A_f = 1$, so $|F| = n-1$. Otherwise, $n \leq 2t-2$. In this case, the number of functions $f\in F$ such that $A_f = \Omega_f =1$ is $2t-n-1$, the number of $f\in F$ such that $A_f \geq 2$ is $n-t$, and the number of $f\in F$ such that $\Omega_f \geq 2$ is $n-t$, so $|F| = n-1$.

 Now label $\mathbf{n} = \{x_1,x_2,\dots ,x_n\}$ and $\mathbf{t} = \{y_1,y_2,\dots ,y_t\}$ in order. For each $i\in[n-1]$, we must show that there are $t-1$ functions $f\in F$ that separate $x_i$ and $x_{i+1}$. This is the same as showing that there are $n-t$ functions $f\in F$ such that $f(x_i) = f(x_{i+1})$. For each $i\in[n-1]$ and each $j\in[t]$, let $F_{i,j}$ be the number of functions $f\in F$ such that $f(x_i) = f(x_{i+1}) = y_j$. First, observe that
\begin{align}
F_{i,1} =
\begin{cases}
n-t+1-i &\text{if } i\leq n-t,\\
0 &\text{otherwise},
\end{cases}
\end{align}
and that 
\begin{align}
F_{i,t} =
\begin{cases}
i-t+1 &\text{if } i\geq t,\\
0 &\text{otherwise}.
\end{cases}
\end{align}
For $2\leq j \leq t-1$, it's easy to see that, if there is an $f\in F$ such that $f(x_i) = f(x_{i+1}) = y_j$, then it is unique. Therefore $F_{i,j}$ is either $1$ or $0$, and
\begin{align}
F_{i,j} =
\begin{cases}
1 &\text{if } j \leq i \leq n-t+j-1,\\
0 &\text{otherwise}.
\end{cases}
\end{align}
Finally, for each $i\in[n-1]$, the number of functions $f\in F$ such that $f(x_i) = f(x_{i+1})$ is equal to $\sum\limits_{j=1}^t F_{i,j}$. To compute this sum, we need to consider four cases. If $t \leq i \leq n-t$, then
\begin{equation}\begin{split}
\sum\limits_{j=1}^t F_{i,j} = n-2t+2+\left|\{j : 2\leq j \leq t-1, F_{i,j} = 1\}\right| = \\
n-2t+2 + (t-2) = n-t.
\end{split}\end{equation}
If $i \leq n-t$ and $i \leq t-1$, then
\begin{equation}\begin{split}
\sum\limits_{j=1}^t F_{i,j} = n-t+1-i+\left|\{j : 2\leq j \leq t-1, F_{i,j} = 1\}\right| = \\
n-t+1-i + (i-1) = n-t.
\end{split}\end{equation}
If $i \geq n-t+1$ and $i \geq t$, then
\begin{equation}\begin{split}
\sum\limits_{j=1}^t F_{i,j} = n-t+1-i+\left|\{j : 2\leq j \leq t-1, F_{i,j} = 1\}\right| = \\
i-t+1 + (n-i-1) = n-t.
\end{split}\end{equation}
If $n - t +1 \leq  i \leq  t-1$, then
\begin{equation}
\sum\limits_{j=1}^t F_{i,j} = \left|\{j : 2\leq j \leq t-1, F_{i,j} = 1\}\right| =
i - (i-n+t+1)+1 = n-t.
\end{equation}
In all four cases, $\left| \{f\in F : f(x_i) = f(x_{i+1})\} \right| = n-t$, so there are $t-1$ functions in $F$ that separate $x_i$ from $x_{i+1}$.
\end{proof}

 We now define a concept that will be useful in proving lower bounds on fractional local $t$--dimension.
Given a poset $P$ and integer $t\geq 2$, a \emph{fractional local $t$--antirealiser} of $P$ is an ordered pair of functions $(I,D)$, where $I : \{(x,y)\in P^2 : x \not\geq y\} \to [0,1]$ and $D : P \to [0,1]$, such that $\sum_{x\in P} D(x) = 1$ and, for each monotone partial function $f : P \to \mathbf{t}$,
\begin{equation}
\sum\limits_{f(x) < f(y)} I(x,y) \leq \sum\limits_{x\in \dom f} D(x).
\end{equation}
A fractional local $t$--antirealiser can be thought of an an obstacle to constructing a fractional local $t$--realiser with small local weight at every point. It can be shown using the strong linear programming duality theorem that $\fltdim{t}{P}$ is equal to the maximum of $\sum\limits_{x\not\geq y} I(x,y)$ over all fractional local $t$--antirealisers $(I,D)$ of $P$.


 Other types of fractional antirealisers can be defined in a similar way; for example, we can define a fractional $t$--antirealiser of $P$ as a function  $I : \{(x,y)\in P^2 : x \not\geq y\} \to [0,1]$ such that $\sum_{f(x) < f(y)} I(x,y) \leq 1$ for all monotone total functions $f:P\to\mathbf{t}$. The fractional $t$--dimension of $P$ is then equal to the maximum of $\sum_{x\not\geq y}I(x,y)$ over all $t$--antirealisers $I$ of $P$. In fact, we have already used fractional $t$--antirealisers implicitly in the proof of the lower bound in Theorem~\ref{thm:ftdimchain}.

\begin{proposition}\label{prop:fltdimantichain}
For all $t\geq 2$, $\ftdim{t}{A_n} = \frac{2t}{t-1}-o(1)$ as $n\to\infty$, and the same is true of $\fltdim{t}{A_n}$.
\end{proposition}
\begin{proof}
As mentioned earlier, the upper bound $\ftdim{t}{A_n} \leq \frac{2t}{t-1}$ is trivial. In fact, we need only assign positive weight to nonconstant functions, so
\begin{equation}
\ftdim{t}{A_n} \leq \frac{2t}{t-1}\cdot t^{-n}(t^n-t) = \frac{2t}{t-1}\left(1-t^{1-n}\right).
\end{equation}

 Let $D(x) = \frac{1}{n}$ for all $x\in A_n$ and let $I(x,y) = \frac{2t}{(t-1)n^2}$ for all $x\neq y$.
Suppose $f$ is a partial function from $A_n$ to $\mathbf{t}$ whose domain has $k$ elements. Then $f$ separates at most $\frac{t-1}{2t}k^2$ ordered pairs (i.e., the number of edges in a $t$-partite Tur\'an graph on $k$ vertices), so
\begin{equation}
\sum\limits_{f(x)<f(y)} I(x,y) \leq \frac{t-1}{2t}k^2\cdot \frac{2t}{(t-1)n^2} = \frac{k^2}{n^2} \leq \frac{k}{n}.
\end{equation}
Therefore $(I,D)$ is a fractional local $t$--antirealiser of $A_n$, so
\begin{equation}
\fltdim{t}{A_n} \geq \frac{2t}{t-1}\left(1-\frac{1}{n}\right).
\end{equation}
\end{proof}


\subsection{Fractional local \texorpdfstring{$t$}{t}-dimension of chains}

 Unlike the other dimension variants, determining the fractional local $t$--dimension of a chain is not trivial, and in general we are unable to determine the exact value of $\fltdim{t}{\mathbf{n}}$. An argument similar to the proof of Proposition~\ref{prop:ltdimchain} shows that, for all integers $t\geq 2$ and $n\in \N$, $\fltdim{t}{\mathbf{tn}} \leq \fltdim{t}{\mathbf{n}}+1$. Therefore an improvement over the trivial bound $\fltdim{t}{\mathbf{n}} \leq \lceil\log_t n\rceil$ for any chain automatically yields an improvement (by an additive constant) for all chains.

 The smallest $n$ such that $\fltdim{2}{\mathbf{n}} < \ltwodim{\mathbf{n}}$ is $5$. Indeed, the fractional local $2$--antirealiser shown in Figure~\ref{fig3} shows that $\fltwodim{\mathbf{3}}$ is at least $2$, and by the trivial bound $\fltwodim{\mathbf{3}} \leq \ltwodim{\mathbf{3}} = 2$, it is exactly $2$. By monotonicity, the same is true for $\fltwodim{\mathbf{4}}$. The fractional local $2$--antirealiser in Figure~\ref{fig4} shows that $\fltwodim{\mathbf{5}} \geq \frac{5}{2}$.
 \pagebreak

\begin{figure}[htbp]
\centering
\begin{tikzpicture}
[
  fsnode/.style={fill=black},
  ssnode/.style={fill=black},
  every fit/.style={ellipse,draw,inner sep=2pt,text width=0mm},
  shorten >= 2pt,shorten <= 2pt, auto, thick]

{
[start chain=going right,node distance=22mm]
\foreach \a in {1,2,...,3}{
\node[fsnode,on chain,draw,circle,inner sep=2pt] (f\a) {};
}
}

\draw[black](f1) -- (f2) node [midway] {$1$};
\draw[black](f2) -- (f3) node [midway] {$1$};
\draw[black](f2) node [text width=1mm,yshift=-15, fill=white] {$1$};
\end{tikzpicture}

\caption{A fractional local $2$--antirealiser of $\mathbf{3}$.}
\label{fig3}
\end{figure}

\begin{figure}[htbp]
\centering
\begin{tikzpicture}
[
  fsnode/.style={fill=black},
  ssnode/.style={fill=black},
  shorten >= 2pt,shorten <= 2pt, auto, thick]

{
[start chain=going right,node distance=22mm]
\foreach \a in {1,2,...,5}{
\node[fsnode,on chain,draw,circle,inner sep=2pt] (f\a) {};
}
}

\draw[black](f1) -- (f2) node [midway,yshift=-.5mm] {$\tfrac{1}{2}$};
\draw[black](f2) -- (f3) node [midway,yshift=-.5mm] {$\tfrac{1}{2}$};
\draw[black](f3) -- (f4) node [midway,yshift=-.5mm] {$\tfrac{1}{2}$};
\draw[black](f4) -- (f5) node [midway,yshift=-.5mm] {$\tfrac{1}{2}$};\draw[black](f1) -- (f2);
\draw[black](f2) -- (f3);
\draw[black](f3) -- (f4);
\draw[black](f4) -- (f5);
\draw[black](f2) node [text width=1mm,yshift=-15, fill=white] {$\frac{1}{2}$};
\draw[black](f4) node [text width=1mm,yshift=-15, fill=white] {$\frac{1}{2}$};

\draw[black](f3) node [text width=1mm,yshift=11mm, fill=white] {$\frac{1}{2}$};
\draw[bend right=30] (f2) to [out=70, in=110] (f4);

\end{tikzpicture}

\caption{A fractional local $2$--antirealiser of $\mathbf{5}$.}
\label{fig4}
\end{figure}

\noindent
The following fractional local $2$--realiser covers each point with total weight at most $\frac{5}{2}$, showing that $\fltwodim{\mathbf{5}} = \frac{5}{2}$. We denote by $w(ab\dots c \ xy\dots z)$ the weight of the partial monotone function that sends $a,b,\dots c$ to $1$ and $x,y,\dots z$ to $2$.

\begin{align*}
w(123 \ 45) &= \tfrac{1}{2} \\
w(12 \ 345) &= \tfrac{1}{2} \\
w(12 \ 3) &= \tfrac{1}{2} \\
w(3 \ 45) &= \tfrac{1}{2} \\
w(1 \ 2) &= 1\\
w(4 \ 5) &= 1.
\end{align*}
\noindent
It follows that $\fltwodim{\mathbf{n}} \leq \lceil\log \frac{n}{5}\rceil + \frac{5}{2}$ for all $n\in \N$.

 To find lower bounds on the fractional local $t$--dimension of chains, we reformulate the problem as follows. Suppose $n\in\N$, and consider the complete graph $K_n$ with vertex set $[n]$.
Given a natural number $t\geq 2$, an ordered $t$-partite graph is a complete $t$-partite subgraph of $K_n$ whose parts can be ordered so that every element of the first part is less than every element of the second part, every element of the second is less than every element of the third, and so on. Then $\fltdim{t}{\mathbf{n}}$ is equal to the maximum value of $\sum_{e\in [n]^{(2)}} I(e)$ over all pairs of functions $D: [n]\to[0,1]$ and $I : [n]^{(2)}\to [0,1]$ such that $\sum_{v\in [n]}D(v) = 1$ and, for every ordered $t$-partite graph $G$, $\sum_{e\in E(G)}I(e)\leq \sum_{v\in V(G)} D(v)$.

 Given an edge $xy\in E(K_n)$, the \emph{length} of $xy$, denoted $\operatorname{length}(xy)$, is $|x-y|$. For each $\ell\in[n-1]$, $K_n$ contains $n-\ell$ edges of length $\ell$. An ordered $t$-partite graph contains at most $(t-1)\ell$ edges of length $\ell$.

 We will prove a lower bound on the fractional local $t$-dimension of chains using the following observation by Hunter Spink~\cite{hunter}. Let $f:[n-1] \to \R$ be a monotone decreasing function and let $B$ be an ordered $t$-partite graph with $k$ vertices. We claim that $\sum_{e\in E(B)} f(\operatorname{length}(e))$ is maximised when $B$ is compressed (i.e., the vertex set of $B$ is a contiguous subset of $[n]$) and $B$ is a Tur\'an graph. To prove the first claim, take the largest contiguous set of vertices in $B$ containing the leftmost vertex, and move all the vertices in this set one step to the right. Observe that this does not increase the length of any edge in $B$. Repeat this process until $B$ is compressed. For the second claim, assume $B$ is compressed. Label the parts of $B$ $B_1,B_2,\dots,B_t$ in order and suppose that $|B_{i}| > |B_{i+1}|$. Let $v$ be the last vertex of $B_i$. If we move $v$ to $B_{i+1}$, we lose an edge of length $\ell$ for each $\ell\in\big[|B_{i+1}|\big]$ and gain an edge of length $k$ for each $i\in[\big|B_{i}\|-1\big]$. Since $\big[|B_i|-1\big] \subseteq \big[|B_{i+1}|\big]$, $\sum_{e\in E(B)} f(\operatorname{length}(e))$ does not increase, and, by repeating this process, we can transform $B$ into a Tur\'an graph.

 Before proving the theorem, we need to introduce some notation. Given an integer $a\geq 0$, the $a^\textrm{th}$ harmonic number, denoted $H_a$, is equal to $\sum_{k=1}^a \frac{1}{k}$. For all $a\in\N$, $H_a \geq \ln a + \gamma$, where $\gamma$ is Euler's constant. If $1 \leq a \leq b$, then
\begin{equation}\label{eq:harmonicdifference}
H_b - H_a = \sum\limits_{k=a+1}^b \frac{1}{k} = \int_a^b\frac{dx}{\lceil x \rceil}\leq \int_a^b\frac{dx}{x} = \ln b - \ln a.
\end{equation}

\begin{theorem}\label{thm:fltdimchain}
For every integer $t\geq 2$,
\[\fltdim{t}{\mathbf{n}} \geq \log_{\sqrt{e}\cdot t} n - O_t(1)\]
as $n\to\infty$.
\end{theorem}

\begin{proof}
Let $D(v) = \frac{1}{n}$ for each $v\in[n]$. For each edge $e\in E(K_n)$, let $I(e) = \frac{2}{(2\ln t+1)n}\cdot \frac{1}{\operatorname{length}(e)}$.

 Suppose $B$ is an ordered $t$-partite subgraph of $K_n$ with $k$ vertices. We want to show that $\sum_{e\in E(B)} I(e) \leq \sum_{v\in V(B)} D(v)$. By the above observation and the fact that $\sum_{v\in V(B)} D(v)$ depends only on $k$, we may assume that $B$ is a compressed Tur\'an graph. Write $k = tq+r$, where $q$ and $r$ are integers and $0\leq r\leq t-1$. For each natural number $\ell$, the number of edges of length $\ell$ in $E(B)$ is equal to the number of edges of length $\ell$ in a compressed $K_k$, minus the number of edges of length $\ell$ in $t-r$ compressed copies of $K_{q}$ and $r$ compressed copies of $K_{q+1}$. For $1\leq \ell \leq q-1$, the number of edges of length $\ell$ is therefore
\begin{equation}\label{eq:shortedges}
tq+r-\ell -(t-r) (q-\ell) - r(q+1-\ell) =
(t-1)\ell
.
\end{equation}
The number of edges of length $q$ is
\begin{equation}\label{eq:lengthqedges}
tq+r-q - r = (t-1)q.
\end{equation}
For $q+1 \leq \ell \leq k$, the number of edges of length $\ell$ is
\begin{equation}\label{eq:longedges}
tq+r-\ell.
\end{equation}
Now let $c = \frac{2}{(2\ln t+1)n}$. It follows from equations~\ref{eq:shortedges},~\ref{eq:lengthqedges}, and~\ref{eq:longedges} that
\begin{equation}\begin{split}
\sum\limits_{e\in E(B)} I(e) = c\sum\limits_{\ell=1}^{q}\frac{(t-1)\ell}{\ell} + c\sum\limits_{\ell=q+1}^{k}\frac{tq+r-\ell}{\ell}
= \\
c\big( (t-1)q + k(H_{k} - H_{q}) - (tq+r-q) \big) = \\
c\big( k(H_{k} - H_q) - r \big)
.
\end{split}\end{equation}
If $q \geq 1$, then, by inequality~\ref{eq:harmonicdifference},
\begin{equation}\begin{split}
c\big( k(H_k - H_q)-r \big) \leq c\left( k\ln \frac{tq+r}{q} - r \right) \leq \\
c\left( k\ln t + k\frac{r}{tq} - r \right) = c\left( k\ln t + \frac{r^2}{tq} \right)
,
\end{split}\end{equation}
and, since $k > 2r$,
\begin{equation}\begin{split}
c\left( k\ln t + \frac{r^2}{tq} \right)
<
c\left( k\ln t + \frac{r}{q} \right) \leq c\left( k\ln t + r \right)
< \\
c\left( k\ln t + \frac{k}{2} \right)
\leq
ck\left( \ln t + \frac{1}{2} \right) = \frac{k}{n}
.
\end{split}\end{equation}
If $q=0$, then $k = r \leq t-1$ and $H_q = 0$, so
\begin{equation}\begin{split}
c\big( k(H_k - H_q)-r \big) = c\big( k(H_k-1) \big) \leq ck\ln k < ck\ln t \leq \frac{k}{n}
.
\end{split}\end{equation}
In both cases,
$\sum\limits_{e\in E(B)} I(e) \leq \frac{k}{n} = \sum\limits_{v\in V(B)} D(v)$.
Therefore,
\begin{equation}\begin{split}
\fltdim{t}{\mathbf{n}} \geq \sum\limits_{e\in [n]^{(2)}} I(e) =
c\sum\limits_{\ell = 1}^{n} \frac{n-\ell}{\ell} = \\
cn(H_n-1) \geq
\frac{2}{2\ln t+1}(\ln n + \gamma - 1) = \\
\log_{\sqrt{e}\cdot t} n - \frac{2-2\gamma}{2\ln t+1}.
\end{split}\end{equation}
\end{proof}

\subsection{Suborders of the hypercube and posets of bounded degree}
 In this subsection, we consider the fractional $t$-dimension of two-layer suborders of the hypercube.

 Brightwell and Scheinerman~\cite{brightwellscheinerman} proved that, for all $n\in\N$ and $2\leq k\leq n-1$, $\fdim{\cube{n}{1,k}} = k+1$. Smith and Trotter~\cite{smithtrotter} determined the exact value of $\lim\limits_{n\to\infty} \fldim{\cube{n}{1,k}}$ for all $k$, and showed that it is equal to $\frac{k}{\ln k-\ln\ln k -o(1)}$ as $k\to\infty$.

 The following theorem shows that Brightwell and Sheinerman's result is within a constant factor of the correct value for fractional $t$--dimension.
\begin{theorem}\label{thm:frac2layers}
For every integer $k \geq 1$, as $n\to\infty$,
\[\ftwodim{\cube{n}{1,k}} \to \left(1-\tfrac{1}{k+1}\right)^{-k}\cdot(k+1) \leq e(k+1).\]
\end{theorem}
\begin{proof}
Every function $f$ from $[n]$ to $\mathbf{2}$ can be extended to a monotone function $f':\cube{n}{1,k}\to \mathbf{2}$, where $f'(A) = \max\{f(a) : a\in A\}$. Assume $n$ is a multiple of $k+1$ and write $\ell=k+1$, $m = \frac{n}{\ell}$. Define a function $w$ that assigns weight $\binom{\ell (m-1)}{m-1}^{-1}$ to every function $f'$, where $f : [n] \to \mathbf{2}$ and $|f^{-1}\{1\}| = \frac{k}{k+1}n$, and weight $0$ to every other monotone function. For every pair $(A,x)$ where $A\in [n]^{(k)}$ and $x\in [n]\setminus A$, there are $\binom{\ell(m-1)}{m-1}$ functions $f : [n] \to \mathbf{2}$ such that $f'(A) < f'\{x\}$, so the total weight of all such functions is $1$. It's easy to check that all other non-relations are covered with total weight at least $1$. Now the total number of monotone functions with positive weight is $\binom{\ell m}{m}$, so the total weight of all these functions is
\begin{equation}
\begin{split}
\binom{\ell m}{m}\binom{\ell (m-1)}{m-1}^{-1}
=
\frac{(\ell m)!}{m!((\ell-1)m)!} \cdot \frac{(m-1)!((\ell-1)(m-1))!}{(\ell(m-1))!}
= \\
\ell\cdot \frac{(\ell m-1) \cdot (\ell m - 2) \cdot \dots \cdot (\ell(m-1)+1)}{((\ell-1)m) \cdot ((\ell-1)m-1) \cdot \dots \cdot ((\ell-1)(m-1)+1)}
\leq \\
\ell\left(\frac{\ell m}{(\ell-1)(m-1)}\right)^{\ell-1},
\end{split}
\end{equation}
which goes to $\ell\cdot\left(\frac{\ell}{\ell-1}\right)^{\ell-1} = \ell\cdot\left(1-\frac{1}{\ell}\right)^{1-\ell} \leq e\ell$ as $m\to\infty$. Because $\ftwodim{\cube{n}{1,k}} = \ftwodim{\cube{\ell m}{1,\ell-1}}$ is monotone increasing in $m$, we have $\ftwodim{\cube{n}{1,\ell-1}} \leq \ell\cdot\left(1-\frac{1}{\ell}\right)^{1-\ell}$ for all $n$.

 Now, for the lower bound, we will construct a fractional $2$-antirealiser of $\cube{n}{1,k}$. As before, let $\ell = k+1$ and assume $n = \ell m$, where $m$ is an integer. For every pair $(A,x)$ with $A\in [n]^{(k)}$ and $x\in [n]\setminus A$, let $I(A,x) = \frac{k!}{m(km)^k}$. Now suppose $f$ is a montone function from $\cube{n}{1,k}$ to $\mathbf{2}$, and define a function $g:[n] \to \mathbf{2}$, where $g(x) = f\{x\}$. Now, if given $A\in [n]^{(k)}$ and $x\in [n]\setminus A$, if $f(A) < f\{x\}$, then $g'(A) < f\{x\}$. We may therefore assume that $f = g'$ without reducing the number of separated pairs. Now let $p$ be the number of elements $x\in[n]$ such that $g(x) = 2$. The number of pairs $(A,x)$ separated by $g'$ is $p\cdot\binom{n-p}{k} \leq \frac{1}{k!}p(n-p)^k$, and the right side of this inequality is maximised when $p = \frac{n}{k+1} = m$. Hence every monotone function separates at most $\frac{1}{k!}m(km)^k$ pairs, so the sum of $I(A,x)$ over all such pairs is at most $1$. Therefore $I$ is a fractional $2$-antirealiser of $\cube{n}{1,k}$, so
\begin{equation}\begin{split}
\ftwodim{\cube{n}{1,k}} \geq \ell\binom{\ell m}{\ell}\cdot \frac{(\ell-1)!}{m\big((\ell-1)m\big)^{\ell-1}} = \\
\frac{(\ell m)^\ell - O(m^{\ell-1})}{m\big((\ell-1)m\big)^{\ell-1}} =
\ell\cdot \left(\frac{\ell}{\ell-1}\right)^{\ell-1} - O\left(\tfrac{1}{m}\right)
=\\
\left(1-\tfrac{1}{k+1}\right)^{-k}\cdot(k+1) - O\left(\tfrac{1}{n}\right)
.
\end{split}\end{equation}
\end{proof}

 Recall that the \emph{outdegree} of an element $x$ of a poset $P$ is the number of elements of $P$ that are strictly greater than $x$. Using Theorem~\ref{thm:frac2layers}, we can bound the fractional $2$-dimension of any poset by a function of its maximum outdegree. We first need the following lemma.

\begin{lemma}\label{lem:splitfrac}
Let $P$ be a poset and let $Q$ be the split of $P$. Then $\ftwodim{P} \leq \ftwodim{Q}$.
\end{lemma}
\begin{proof}
Let $w$ be a fractional $2$-realiser of $Q$. For each monotone $f : Q \to \mathbf{2}$, define a function $f':P\to\mathbf{2}$, where $f'(x) = \max\left\{ f(y') : y\leq x \right\}$. It's clear that $f'$ is monotone. Now, for each monotone $g:P\to\mathbf{2}$, let $w'(g) = w(f)$ if $g=f'$ for some monotone $f:Q\to\mathbf{2}$ and $w'(g) = 0$ otherwise. Suppose $a\not\geq_P b$. Then $a'' \not\geq_Q b'$, so the total $w$-weight of all montone functions $f$ such that $f(a'') = 0$ and $f(b') = 1$ is at least $1$. For each such $f$, $f'(c') = 0$ for all $c\leq a$, so $f'(a) = 0$, and $f'(b) = 1$. Hence the pair $(a,b)$ is separated with total weight at least $1$, so $w'$ is a fractional $2$-realiser of $P$ with the same total weight as $w$.
\end{proof}

\begin{corollary}\label{cor:MFD}
Let $P$ be a poset with maximum outdegree $\upsilon$. Then $\ftwodim{P} \leq e(\upsilon+2)$.
\end{corollary}
\begin{proof}
Let $Q$ be the split of $P$. By Lemma~\ref{lem:splitfrac}, $\ftwodim{P} \leq \ftwodim{Q}$. Since $Q$ has maximum outdegree $\upsilon+1$, its dual can be embedded into $\cube{n}{1,\upsilon+1}$ for some large $n$. Therefore, by Theorem~\ref{thm:frac2layers}, $\ftwodim{Q} \leq e(\upsilon+2)$.
\end{proof}

\section{Open questions}

 By Theorem~\ref{thm:lowerbound} and Kierstead's theorem, we know that, for fixed $\ell < k$, $\ltdim{t}{\cube{n}{\ell,k}} = \Theta_{t,\ell,k}(\log n)$ as $n\to\infty$. However, the constant factors on the upper and lower bounds are very far apart, and we would like to know if they can be improved.

\begin{question}
Given $1 \leq \ell < k \leq n$ and $t\geq 2$, what is $\ltdim{t}{\cube{n}{\ell,k}}$? In particular, what is $\ltwodim{\cube{n}{1,2}}$?
\end{question}

 The local dimension of $\cube{n}{}$ is still unknown. The best known lower bound is $\Omega\left(\frac{n}{\log n}\right)$, but the only known upper bound is $n$. Maybe studying the local $t$--dimension of $\cube{n}{}$ will help solve this problem.

\begin{question}
What is $\ltdim{t}{\cube{n}{}}$ for $t\geq 3$? Is it ever strictly less than $n$? In general, what is $\ltdim{t}{\mathbf{s}^n}$ when $t > s$?
\end{question}

 The maximum local $t$--dimension of an $n$-element poset is $\Theta\left(\frac{n}{\log n}\right)$, with upper and lower bounds that do not depend on $t$. This leads to the next question.
\begin{question}
What is the maximum local $t$--dimension of an $n$-element poset? Does it depend on $t$?
\end{question}

 Of course, all of the natural questions asked of the other parameters (e.g., the maximum and minimum value for $n$-element posets, the value for the Boolean lattice and for its two-layer suborders, etc.) can be asked of fractional $t$--dimension and fractional local $t$--dimension as well.

 It follows from Theorem~\ref{thm:ftdimchain} that, for every integer $t\geq 2$ and every $n\in\N$, $\ftdim{t}{\mathbf{n}} = \lceil\ftdim{t}{\mathbf{n}}\rceil$. This motivates the following problem.

\begin{problem}
Characterise the posets $P$ for which $\tdim{t}{P} = \lceil\ftdim{t}{P}\rceil$.
\end{problem}


Proposition~\ref{prop:ltdimchain} and Theorem~\ref{thm:fltdimchain} together imply that $\fltdim{t}{\mathbf{n}} = \Theta_t(\log n)$ as $n\to\infty$, for every fixed $t\geq 2$. However, we do not have a formula for the exact fractional local $t$-dimension of a chain.

\begin{problem}
What is the exact value of $\fltdim{t}{\mathbf{n}}$, for all integers $n\geq t \geq 2$?
\end{problem}

By an argument similar to the proof of Inequality 3 of Proposition 2 in \cite{lewis}, for any $t\geq 2$ and all $m,n\in\N$, $\fltdim{t}{\mathbf{mn}} \leq \fltdim{t}{\mathbf{m}} + \fltdim{t}{\mathbf{n}}$. It follows that, if $\fltdim{t}{\mathbf{m}} < \log_t m$ for any $m$, then we can improve the trivial upper bound $\fltdim{t}{\mathbf{n}} \leq \lceil\log_t n\rceil$ by a constant factor for all $n$. However, we do not know of any examples of such $m$ for any $t$.

 One immediate corollary of Theorem~\ref{thm:frac2layers} is that the functions $\operatorname{FD}_t(k) = \lim\limits_{n\to\infty} \ftdim{t}{\cube{n}{1,k}}$ and $\operatorname{FLD}_t(k) = \lim\limits_{n\to\infty} \fltdim{t}{\cube{n}{1,k}}$ are well-defined for every integer $t\geq 2$. Theorem~\ref{thm:frac2layers} establishes the exact value of $\operatorname{FD}_2(k)$, and shows that it is equal to $(e-o(1))(k+1)$ as $k\to\infty$. Brightwell and Scheinerman's results in~\cite{brightwellscheinerman} give a lower bound for $\operatorname{FD}_t(k)$, and Smith and Trotter's results in~\cite{smithtrotter} give a lower bound for $\operatorname{FLD}_t(k)$.

\begin{question}
What is the exact value of $\operatorname{FD}_t(k)$ and $\operatorname{FLD}_t(k)$, for all integers $t\geq 2$ and $k\in\N$?
\end{question}

 Let $\operatorname{MFD}_t(\Delta)$ be the supremum of $\ftdim{t}{P}$ over all posets whose comparability graphs have maximum degree $\Delta$. Similarly, let $\operatorname{MFLD}_t(\Delta)$ be the supremum of $\fltdim{t}{P}$ over all posets whose comparability graphs have maximum degree $\Delta$. It follows from Corollary~\ref{cor:MFD} that these functions are well-defined.

\begin{question}
What is the exact value of $\operatorname{MFD}_t(\Delta)$ and $\operatorname{MFLD}_t(\Delta)$, for all integers $t\geq 2$ and $\Delta\in\N$?
\end{question}

 We hope to solve these problems in the near future.

\bibliography{localtwodim}{}
\bibliographystyle{plain}
\end{document}